\documentclass[10pt]{amsart}
\usepackage{epsfig,amsfonts,amsthm,amssymb,latexsym,amsmath}

\newcommand{\mymod}[3]{#1 \equiv #2 \kern -0.5em \pmod{#3}}
\newcommand{\mynotmod}[3]{#1 \not \equiv #2 \kern -0.6em \pmod{#3}}

\theoremstyle{plain}
\newtheorem{theorem}{Theorem}[section]
\newtheorem{corollary}[theorem]{Corollary}

\newtheorem{example}[theorem]{Example}

\theoremstyle{remark}
\newtheorem{remark}[theorem]{Remark}

\theoremstyle{definition}
\newtheorem{definition}[theorem]{Definition}

\title[Generalization of Gaussian third-order Jacobsthal numbers]{Generalization of Gaussian third-order Jacobsthal numbers and their new families}

\vspace{5pt}

\author{\scriptsize Gamaliel Morales}
\date{}

\begin{document}
\maketitle

\vspace{-20pt}
\begin{center}
{\footnotesize Instituto de Matem\'aticas, Pontificia Universidad Cat\'olica de Valpara\'iso, Blanco Viel 596, Cerro Bar\'on, Valpara\'iso, Chile. \\
E-mail: gamaliel.cerda.m@mail.pucv.cl 
}\end{center}

\begin{abstract}
In this study, we introduce the generalized Gaussian third-order Jacobsthal numbers with arbitrary initial values and discuss two particular cases, namely, Gaussian third-order Jacobsthal and Gaussian modified third-order Jacobsthal numbers. In this paper we discuss several of its algebraic properties such as Binet's formula, partial sum, generating function, negative subscript elements, d'Ocagne's and Cassini's identities. Furthermore, we study and introduce a new generalization of this sequence called $k$-generalized Gaussian third-order Jacobsthal numbers. We present several of its properties and its connection with the generalized third-order Jacobsthal numbers.
\end{abstract}

\medskip
\noindent
\subjclass{\footnotesize {\bf Mathematical subject classification:} 
Primary: 11B37; Secondary: 11B39, 65Q30.}

\medskip
\noindent
\keywords{\footnotesize {\bf Key words:} Binet's formula, Gaussian third-order Jacobsthal number, generating function, partial sum, third-order Jacobsthal number. }
\medskip

\section{Introduction and Motivation}

In 2021, Morales \cite{Mo1} introduced the concept of the Gaussian third-order Jacobsthal and Gaussian modified third-order Jacobsthal numbers defined as
$$
Jg_{n}^{(3)}=J_{n}^{(3)}+iJ_{n-1}^{(3)}
$$
and
$$
Kg_{n}^{(3)}=K_{n}^{(3)}+iK_{n-1}^{(3)},
$$
respectively, where $J_{n}^{(3)}$ is the $n$-th third-order Jacobsthal number and $K_{n}^{(3)}$ is the $n$-th modified third-order Jacobsthal number. Morales \cite{Mo2} also defined a generalization of Gaussian third-order Jacobsthal called Gaussian third-order Jacobsthal polynomials using the algebraic properties of the third-order Jacobsthal polynomials (see \cite{Mo2,Mo3}).

In this paper, we deal with the generalization of third-order Jacobsthal numbers which is given by
$$
J_{n}^{(3)}=\frac{1}{7}\left[2^{n+1}+\Omega_{n}-2\Omega_{n+1}\right],
$$
where $$\Omega_{n}=\frac{1}{\omega_{1}-\omega_{2}}\left[\omega_{1}^{n}-\omega_{2}^{n}\right]=\left\{ 
\begin{array}{ccc}
0 & \textrm{if} & \mymod{n}{0}{3} \\ 
1 & \textrm{if} & \mymod{n}{1}{3} \\ 
-1 & \textrm{if} & \mymod{n}{2}{3}
\end{array}
\right. ,$$
and $\omega_{1}$, $\omega_{2}$ are roots of the equation $x^{2}+x+1=0$, and have many interesting properties.

Note that the generalized third-order Jacobsthal sequence $\{\mathcal{J}_{n}^{(3)}\}_{n\geq0}$ is given by the recurrence relation
\begin{equation}\label{e1}
\mathcal{J}_{n+3}^{(3)}=\mathcal{J}_{n+2}^{(3)}+\mathcal{J}_{n+1}^{(3)}+2\mathcal{J}_{n}^{(3)},\ \mathcal{J}_{0}^{(3)}=a,\ \mathcal{J}_{1}^{(3)}=b,\ \mathcal{J}_{2}^{(3)}=c,
\end{equation}
and the terms of this sequence are known as the generalized third-order Jacobsthal numbers. As a special case of generalized third-order Jacobsthal sequence, setting $a=0$ and $b=c=1$ in equation (\ref{e1}) gives the usual third-order Jacobsthal sequence $\{J_{n}^{(3)}\}_{n\geq0}$ and for $a=3$, $b=1$ and $c=3$, it gives the modified third-order Jacobsthal sequence $\{K_{n}^{(3)}\}_{n\geq0}$. Also, the characteristic equation corresponding to the above recurrence relation is 
\begin{equation}\label{e2}
\xi^{3}-\xi^{2}-\xi-2=0.
\end{equation}
Equation (\ref{e2}) has three roots, $\xi=2$, $\omega_{1}$ and $\omega_{2}$ and they satisfy $\omega_{1}+\omega_{2}=-1$, $\omega_{1}\omega_{2}=1$ and $\omega_{1}-\omega_{2}=i\sqrt{3}$. Thus, the Binet's formula for the generalized third-order Jacobsthal numbers is given by
\begin{equation}\label{bi}
\mathcal{J}_{n}^{(3)}=\frac{1}{7}\left[(a+b+c)\cdot 2^{n}+(4a+4b-3c)\Omega_{n}+(6a-b-c)\Omega_{n+1}\right].
\end{equation}
Some recent developments on third-order Jacobsthal numbers and their applications, can be seen in \cite{Mo4,Mo5,Mo6}

Motivated by the work of Berzsenyi \cite{Be}, Jordan \cite{Jo}, Pethe and Horadam \cite{Pe} on Gaussian numbers and classical sequences of integers such as Fibonacci and their generalizations, we generalize the Gaussian third-order Jacobsthal numbers with arbitrary initial values and give a new family of the generalized Gaussian third-order Jacobsthal numbers. Furthermore, we obtain their algebraic properties such as Binet's formula, generating function, partial sum, d'Ocagne's and Cassini's identities. 

Many other authors have used this technique to study properties of integer sequences. For example, A\c{s}ci and G\"urel \cite{As}, Halici and Oz \cite{Ha}, Kumari et al. \cite{Ku}, Ta\c{s}ci \cite{Ta} and several others.

\section{Generalized Gaussian third-order Jacobsthal numbers}

In this section, we introduce the generalized Gaussian third-order Jacobsthal sequence $\{\mathcal{J}g_{n}^{(3)}\}_{n\geq0}$ and we present some of its algebraic properties, identities and relations with usual Gaussian third-order Jacobsthal and Gaussian modified third-order Jacobsthal numbers.

\begin{definition}\label{def1}
The generalized Gaussian third-order Jacobsthal sequence $\{\mathcal{J}g_{n}^{(3)}\}_{n\geq0}$ is defined by
$$
\mathcal{J}g_{n+3}^{(3)}=\mathcal{J}g_{n+2}^{(3)}+\mathcal{J}g_{n+1}^{(3)}+2\mathcal{J}g_{n}^{(3)},
$$
where $\mathcal{J}g_{0}^{(3)}=a+\frac{1}{2}(c-b-a)i$, $\mathcal{J}g_{1}^{(3)}=b+ai$ and $\mathcal{J}g_{2}^{(3)}=c+bi$. Also, $a$, $b$ and $c$ are arbitrary real numbers not all being zero.
\end{definition}

\begin{remark}
In terms of generalized third-order Jacobsthal numbers, the generalized Gaussian third-order Jacobsthal can be written as $\mathcal{J}g_{n}^{(3)}=\mathcal{J}_{n}^{(3)}+\mathcal{J}_{n-1}^{(3)}i$ for any integer $n\geq1$.
\end{remark}

If $a=0$ and $b=c=1$ in Definition \ref{def1}, we have the Gaussian third-order Jacobsthal sequence $\{Jg_{n}^{(3)}\}_{n\geq0}$ defined as
$$
Jg_{n+3}^{(3)}=Jg_{n+2}^{(3)}+Jg_{n+1}^{(3)}+2Jg_{n}^{(3)},
$$
where $Jg_{0}^{(3)}=0$, $Jg_{1}^{(3)}=1$ and $Jg_{2}^{(3)}=1+i$. Analogously, if $a=3$, $b=1$ and $c=3$ in Definition \ref{def1}, we obtain Gaussian modified third-order Jacobsthal sequence $\{Kg_{n}^{(3)}\}_{n\geq0}$ defined by
$$
Kg_{n+3}^{(3)}=Kg_{n+2}^{(3)}+Kg_{n+1}^{(3)}+2Kg_{n}^{(3)},
$$
where $Kg_{0}^{(3)}=3-\frac{1}{2}i$, $Kg_{1}^{(3)}=1+3i$ and $Kg_{2}^{(3)}=3+i$. The first few generalized Gaussian third-order Jacobsthal, Gaussian third-order Jacobsthal and Gaussian modified third-order Jacobsthal numbers are given in Table \ref{tab1}.

\begin{table}[ht] 
\caption{Generalized Gaussian third-order Jacobsthal numbers.} 
\centering      
\begin{tabular}{lllll}
\hline
$n$ & $\mathcal{J}g_{n}^{(3)}$ & $Jg_{n}^{(3)}$ & $Kg_{n}^{(3)}$   \\ \hline

$0$ & $a+\frac{1}{2}(c-b-a)i$ & $0$ & $3-\frac{1}{2}i$ 
\\ 
$1$ & $b+ai$ & $1$ & $1+3i$ 
\\ 
$2$ & $c+bi$ & $1+i$ & $3+i$  
\\ 
$3$ & $2a+b+c+ci$ & $2+i$ & $10+3i$  
\\
$4$ & $2a+3b+2c+(2a+b+c)i$ & $5+2i$ & $15+10i$  
\\
$5$ & $4a+4b+5c+(2a+3b+2c)i$ & $9+5i$ & $31+15i$  
\\\hline
\end{tabular}
\label{tab1}  
\end{table}

The first result establishes Binet's formula for this new sequence.
\begin{theorem}\label{teo1}
For any integer $n\geq 0$, the Binet's formula for $\mathcal{J}g_{n}^{(3)}$ is given by
$$
\mathcal{J}g_{n}^{(3)}=\frac{1}{7}\left\lbrace
\begin{array}{c}
(a+b+c)\left(1+\frac{i}{2}\right)\cdot 2^{n}\\
+ \left[4a+4b-3c+(2a-5b+2c)i\right]\Omega_{n}\\
+ \left[6a-b-c-(4a+4b-3c)i\right]\Omega_{n+1}
\end{array}
\right\rbrace ,
$$
where $\Omega_{n}$ is as in equation (\ref{bi}).
\end{theorem}
\begin{proof}
Using differential equation theory, the $n$-th term of generalized Gaussian third-order Jacobsthal sequence can be written as 
\begin{equation}\label{ec}
\mathcal{J}g_{n}^{(3)}=A2^{n}+B\Omega_{n}+C\Omega_{n+1},
\end{equation}
where $\Omega_{n}=\frac{1}{\omega_{1}-\omega_{2}}\left[\omega_{1}^{n}-\omega_{2}^{n}\right]$. Solving for $n=0, 1, 2$ and using the initial values for generalized Gaussian third-order Jacobsthal numbers $\mathcal{J}g_{0}^{(3)}=a+\frac{1}{2}(c-b-a)i$, $\mathcal{J}g_{1}^{(3)}=b+ai$ and $\mathcal{J}g_{2}^{(3)}=c+bi$ in equation (\ref{ec}), we have
\begin{align*}
A+C&=a+\frac{1}{2}(c-b-a)i,\\
2A+B-C&=b+ai,\\
4A-B&=c+bi.
\end{align*}
Solving the above system, we have $7A=(a+b+c)\left(1+\frac{i}{2}\right)$, $7B=4a+4b-3c+(2a-5b+2c)i$ and $7C=6a-b-c-(4a+4b-3c)i$. Then, we obtain
$$
\mathcal{J}g_{n}^{(3)}=\frac{1}{7}\left\lbrace
\begin{array}{c}
(a+b+c)\left(1+\frac{i}{2}\right)\cdot 2^{n}\\
+ \left[4a+4b-3c+(2a-5b+2c)i\right]\Omega_{n}\\
+ \left[6a-b-c-(4a+4b-3c)i\right]\Omega_{n+1}
\end{array}
\right\rbrace .
$$
This completes the proof.
\end{proof}

As special cases, the Binet formulae for Gaussian third-order Jacobsthal and Gaussian modified third-order Jacobsthal numbers are given by
$$
Jg_{n}^{(3)}=\frac{1}{7}\left[\left(1+\frac{i}{2}\right)\cdot 2^{n+1}+(1-3i)\Omega_{n}-(2+i)\Omega_{n+1}\right]
$$
and
$$
Kg_{n}^{(3)}=\left(1+\frac{i}{2}\right)\cdot 2^{n}+(1+i)\Omega_{n}+(2-i)\Omega_{n+1},
$$
respectively.

Two important relationships satisfy the generalized Gaussian third-order Jacobstahl and will be tested in the following result.
\begin{theorem}
For any integer $n\geq 0$, the following identities are satisfied
\begin{equation}\label{teo2:1}
\mathcal{J}g_{n+3}^{(3)}=\mathcal{J}g_{n}^{(3)}+(a+b+c)\left(1+\frac{i}{2}\right)\cdot 2^{n},
\end{equation}
\begin{equation}\label{teo2:2}
\mathcal{J}g_{n}^{(3)}+\mathcal{J}g_{n+1}^{(3)}+\mathcal{J}g_{n+2}^{(3)}=(a+b+c)\left(1+\frac{i}{2}\right)\cdot 2^{n}
\end{equation}
and
\begin{equation}\label{teo2:3}
\mathcal{J}g_{n+1}^{(3)}=2\mathcal{J}g_{n}^{(3)}-\left\lbrace 
\begin{array}{c} \left[2a+b-c+(2b-c)i\right]\Omega_{n}\\
+\left[2a-b-(2a+b-c)i\right]\Omega_{n+1}
\end{array}
\right\rbrace .
\end{equation}
\end{theorem}
\begin{proof}
(\ref{teo2:1}): Using the Binet formula in Theorem \ref{teo1} and $\Omega_{ n+3}=\Omega_{n}$, we have
\begin{align*}
\mathcal{J}g_{n}^{(3)}+(a+b+c)\left(1+\frac{i}{2}\right)\cdot 2^{n}&=\frac{1}{7}\left\lbrace
\begin{array}{c}
(a+b+c)\left(1+\frac{i}{2}\right)\cdot 2^{n}\\
+ \left[4a+4b-3c+(2a-5b+2c)i\right]\Omega_{n}\\
+ \left[6a-b-c-(4a+4b-3c)i\right]\Omega_{n+1}
\end{array}
\right\rbrace \\
&\ \ + (a+b+c)\left(1+\frac{i}{2}\right)\cdot 2^{n}\\
&=\frac{1}{7}\left\lbrace
\begin{array}{c}
(a+b+c)\left(1+\frac{i}{2}\right)\cdot 2^{n}\\
+ 7(a+b+c)\left(1+\frac{i}{2}\right)\cdot 2^{n}\\
+ \left[4a+4b-3c+(2a-5b+2c)i\right]\Omega_{n}\\
+ \left[6a-b-c-(4a+4b-3c)i\right]\Omega_{n+1}
\end{array}
\right\rbrace \\
&=\frac{1}{7}\left\lbrace
\begin{array}{c}
(a+b+c)\left(1+\frac{i}{2}\right)\cdot 2^{n+3}\\
+ \left[4a+4b-3c+(2a-5b+2c)i\right]\Omega_{n+3}\\
+ \left[6a-b-c-(4a+4b-3c)i\right]\Omega_{n+4}
\end{array}
\right\rbrace\\
&=\mathcal{J}g_{n+3}^{(3)}.
\end{align*}
Using equation (\ref{teo2:1}) and Theorem \ref{teo1}, the identities (\ref{teo2:2}) and (\ref{teo2:3}) can be proved.
\end{proof}

\begin{definition}
The generalized Gaussian third-order Jacobsthal numbers with negative subscript $\{\mathcal{J}g_{-n}^{(3)}\}_{n\geq0}$ are defined recursively as
$$
\mathcal{J}g_{-(n+3)}^{(3)}=\frac{1}{2}\left[-\mathcal{J}g_{-(n+2)}^{(3)}-\mathcal{J}g_{-(n+1)}^{(3)}+\mathcal{J}g_{-n}^{(3)}\right],
$$
with initial values $\mathcal{J}g_{0}^{(3)}=a+\frac{1}{2}(c-b-a)i$, $\mathcal{J}g_{-1}^{(3)}=\frac{1}{2}(c-b-a)-\frac{1}{4}(a-3b+c)i$ and $\mathcal{J}g_{-2}^{(3)}=-\frac{1}{4}(a-3b+c)+\frac{1}{8}(7a-b-c)i$.
\end{definition}

\begin{theorem}
For any integer $n\geq 1$, the Binet formulae for the generalized Gaussian third-order Jacobsthal numbers with negative subscripts are given as
\begin{align*}
\mathcal{J}g_{-n}^{(3)}&=\frac{1}{7}\left\lbrace
\begin{array}{c}
(a+b+c)\left(1+\frac{i}{2}\right)\cdot \left(\frac{1}{2}\right)^{n}\\
- \left[4a+4b-3c+(2a-5b+2c)i\right]\Omega_{n}\\
- \left[6a-b-c-(4a+4b-3c)i\right]\Omega_{n-1}
\end{array}
\right\rbrace \\
&=\frac{1}{7}\left\lbrace
\begin{array}{c}
(a+b+c)\left(1+\frac{i}{2}\right)\cdot \left(\frac{1}{2}\right)^{n}\\
+ \left[2a-5b+2c-(6a-b-c)i\right]\Omega_{n}\\
+ \left[6a-b-c-(4a+4b-3c)i\right]\Omega_{n+1}
\end{array}
\right\rbrace .
\end{align*}
\end{theorem}
\begin{proof}
Replacing $n$ by $-n$ in the Binet formula for generalized Gaussian third-order Jacobsthal numbers (see Theorem \ref{teo1}) and $\Omega_{-n}=-\Omega_{n}$, we get the required results. 
\end{proof}

The following result proves the identity of d'Ocagne for the generalized Gaussian third-order Jacobsthal numbers.
\begin{theorem}[d'Ocagne's identity]\label{doc}
For any positive integers $n$ and $m$ such that $n\geq m$, we have
$$
\mathcal{J}g_{m+1}^{(3)}\mathcal{J}g_{n}^{(3)}-\mathcal{J}g_{m}^{(3)}\mathcal{J}g_{n+1}^{(3)}=\frac{1}{49}\left\lbrace
\begin{array}{c}
(\lambda_{1}^{2}-\lambda_{1}\lambda_{2}+\lambda_{2}^{2})\Omega_{n-m}\\
+\lambda\left(1+\frac{i}{2}\right)\left[2^{m}\Psi_{n}-2^{n}\Psi_{m}\right]
\end{array}
\right\rbrace ,
$$
where $\lambda=a+b+c$, $\lambda_{1}=4a+4b-3c+(2a-5b+2c)i$, $\lambda_{2}=6a-b-c-(4a+4b-3c)i$ and $\Psi_{n}=(2\lambda_{1}+\lambda_{2})\Omega_{n}+(3\lambda_{2}-\lambda_{1})\Omega_{n+1}$.
\end{theorem}
\begin{proof}
Using the notation $\lambda_{1}=4a+4b-3c+(2a-5b+2c)i$ and $\lambda_{2}=6a-b-c-(4a+4b-3c)i$, we have $Z_{n}=\lambda_{1}\Omega_{n}+ \lambda_{2}\Omega_{n+1}$ and the next relation
\begin{align*}
Z_{m+1}Z_{n}-Z_{m}Z_{n+1}&=\left[\lambda_{1}\Omega_{m+1}+ \lambda_{2}\Omega_{m+2}\right]\left[\lambda_{1}\Omega_{n}+ \lambda_{2}\Omega_{n+1}\right]\\
&\ \ - \left[\lambda_{1}\Omega_{m}+ \lambda_{2}\Omega_{m+1}\right]\left[\lambda_{1}\Omega_{n+1}+ \lambda_{2}\Omega_{n+2}\right]\\
&=\left[-\lambda_{2}\Omega_{m}+ (\lambda_{1}-\lambda_{2})\Omega_{m+1}\right]\left[\lambda_{1}\Omega_{n}+ \lambda_{2}\Omega_{n+1}\right]\\
&\ \ - \left[\lambda_{1}\Omega_{m}+ \lambda_{2}\Omega_{m+1}\right]\left[-\lambda_{2}\Omega_{n}+(\lambda_{1}-\lambda_{2})\Omega_{n+1}\right]\\
&=(\lambda_{1}^{2}-\lambda_{1}\lambda_{2}+\lambda_{2}^{2})\left[\Omega_{m+1}\Omega_{n}-\Omega_{m}\Omega_{n+1}\right]\\
&=(\lambda_{1}^{2}-\lambda_{1}\lambda_{2}+\lambda_{2}^{2})\Omega_{n-m}.
\end{align*}
Then, we can rewrite $\mathcal{J}g_{n}^{(3)}=\frac{1}{7}\left[
\lambda\left(1+\frac{i}{2}\right)\cdot 2^{n}+Z_{n}
\right]$ and using this notation we get
\begin{align*}
\mathcal{J}g_{m+1}^{(3)}\mathcal{J}g_{n}^{(3)}&-\mathcal{J}g_{m}^{(3)}\mathcal{J}g_{n+1}^{(3)}\\
&=\frac{1}{49}\left[
\lambda\left(1+\frac{i}{2}\right)\cdot 2^{m+1}+Z_{m+1}
\right]\left[
\lambda\left(1+\frac{i}{2}\right)\cdot 2^{n}+Z_{n}
\right]\\
&\ \ - \frac{1}{49}\left[
\lambda\left(1+\frac{i}{2}\right)\cdot 2^{m}+Z_{m}
\right]\left[
\lambda\left(1+\frac{i}{2}\right)\cdot 2^{n+1}+Z_{n+1}
\right]\\
&=\frac{1}{49}\left\lbrace
\begin{array}{c}
Z_{m+1}Z_{n}-Z_{m}Z_{n+1}\\
+\lambda\left(1+\frac{i}{2}\right)\left[2^{m}(2Z_{n}-Z_{n+1})-2^{n}(2Z_{m}-Z_{m+1})\right]
\end{array}
\right\rbrace \\
&=\frac{1}{49}\left\lbrace
\begin{array}{c}
(\lambda_{1}^{2}-\lambda_{1}\lambda_{2}+\lambda_{2}^{2})\Omega_{n-m}\\
+\lambda\left(1+\frac{i}{2}\right)\left[2^{m}\Psi_{n}-2^{n}\Psi_{m}\right]
\end{array}
\right\rbrace ,
\end{align*}
where $\lambda=a+b+c$ and $\Psi_{n}=2Z_{n}-Z_{n+1}=(2\lambda_{1}+\lambda_{2})\Omega_{n}+(3\lambda_{2}-\lambda_{1})\Omega_{n+1}$. This proves the requested result.
\end{proof}

Substituting $m=n-1$ in the d'Ocagne's identity in Theorem \ref{doc} gives the Cassini's identity for the generalized Gaussian third-order Jacobsthal numbers and hence the following theorem.

\begin{theorem}[Cassini's identity]\label{cas}
For any integer $n\geq 1$, we have
$$
\left[\mathcal{J}g_{n}^{(3)}\right]^{2}-\mathcal{J}g_{n-1}^{(3)}\mathcal{J}g_{n+1}^{(3)}=\frac{1}{49}\left\lbrace
\begin{array}{c}
\lambda_{1}^{2}-\lambda_{1}\lambda_{2}+\lambda_{2}^{2}\\
+\lambda\left(1+\frac{i}{2}\right)2^{n-1}\left[\Psi_{n}-2\Psi_{n-1}\right]
\end{array}
\right\rbrace ,
$$
where $\lambda=a+b+c$, $\lambda_{1}=4a+4b-3c+(2a-5b+2c)i$, $\lambda_{2}=6a-b-c-(4a+4b-3c)i$ and $\Psi_{n}=(2\lambda_{1}+\lambda_{2})\Omega_{n}+(3\lambda_{2}-\lambda_{1})\Omega_{n+1}$.
\end{theorem}

Developing the right component in the previous result, we obtain
\begin{align*}
\Psi_{n}-2\Psi_{n-1}&=(2\lambda_{1}+\lambda_{2})\Omega_{n}+(3\lambda_{2}-\lambda_{1})\Omega_{n+1}\\
&\ \ - 2(2\lambda_{1}+\lambda_{2})\Omega_{n-1}-2(3\lambda_{2}-\lambda_{1})\Omega_{n}\\
&=(2\lambda_{1}+\lambda_{2})\Omega_{n}+(3\lambda_{2}-\lambda_{1})\Omega_{n+1}\\
&\ \ -2(2\lambda_{1}+\lambda_{2})\left[-\Omega_{n}-\Omega_{n+1}\right]-2(3\lambda_{2}-\lambda_{1})\Omega_{n}\\
&=(8\lambda_{1}-3\lambda_{2})\Omega_{n}+(3\lambda_{1}+5\lambda_{2})\Omega_{n+1}.
\end{align*}
Then, we can write Theorem \ref{cas} as
$$
\left[\mathcal{J}g_{n}^{(3)}\right]^{2}-\mathcal{J}g_{n-1}^{(3)}\mathcal{J}g_{n+1}^{(3)}=\frac{1}{49}\left\lbrace
\begin{array}{c}
\lambda_{1}^{2}-\lambda_{1}\lambda_{2}+\lambda_{2}^{2}\\
+\lambda\left(1+\frac{i}{2}\right)2^{n-1}\left[(8\lambda_{1}-3\lambda_{2})\Omega_{n}+(3\lambda_{1}+5\lambda_{2})\Omega_{n+1}\right]
\end{array}
\right\rbrace .
$$

In Table \ref{tab2}, we can see the values of $\lambda$, $\lambda_{1}$ and $\lambda_{2}$ according to the initial conditions of the Gaussian third-order Jacobsthal and Gaussian modified third-order Jacobsthal numbers.

\begin{table}[ht] 
\caption{Values of $\lambda$, $\lambda_{1}$ and $\lambda_{2}$.} 
\centering      
\begin{tabular}{lllll}
\hline
 & $\mathcal{J}g_{n}^{(3)}$ & $\mathcal{J}g_{n}^{(3)}$ & $Kg_{n}^{(3)}$ \\ \hline
$\lambda$ & $a+b+c$ & $2$ & $7$ 
\\ 
$\lambda_{1}$ & $4a+4b-3c+(2a-5b+2c)i$ & $1-3i$ & $7(1+i)$  
\\ 
$\lambda_{2}$ & $6a-b-c-(4a+4b-3c)i$ & $-2-i$ & $7(2-i)$   
\\ 
$\Psi_{n}$ & $(2\lambda_{1}+\lambda_{2})\Omega_{n}+(3\lambda_{2}-\lambda_{1})\Omega_{n+1}$ & $-7\left[i\Omega_{n}+\Omega_{n+1}\right]$ & $7\left[(4+i)\Omega_{n}+(5-4i)\Omega_{n+1}\right]$   
\\\hline
\end{tabular}
\label{tab2}  
\end{table}

As special case of the above theorem, we deduce the following corollary.
\begin{corollary}
For any integer $n\geq 1$, the following identities are verified
$$
\left[Jg_{n}^{(3)}\right]^{2}-Jg_{n-1}^{(3)}Jg_{n+1}^{(3)}=\frac{1}{7}\left\lbrace 
-i+\left(1+\frac{i}{2}\right)2^{n}\left[(2-3i)\Omega_{n}-(1+2i)\Omega_{n+1}\right] \right\rbrace
$$
and
$$
\left[Kg_{n}^{(3)}\right]^{2}-Kg_{n-1}^{(3)}Kg_{n+1}^{(3)}=-3i
+\left(1+\frac{i}{2}\right)2^{n-1}\left[(2+5i)\Omega_{n}+(13-2i)\Omega_{n+1}\right].
$$
\end{corollary}

\begin{theorem}[Generating function for $\mathcal{J}g_{n}^{(3)}$]\label{gen}
For $1-\xi-\xi^{2}-2\xi^{3}\neq 0$, the generating function for the generalized Gaussian third-order Jacobsthal numbers is given by
$$
g\left(\mathcal{J}g_{n}^{(3)};\xi\right)=\frac{\left\lbrace
\begin{array}{c}
a+(b-a)\xi+(c-b-a)\xi^{2}\\
+\frac{1}{2}\left[c-b-a+(3a+b-c)\xi-(a-3b+c)\xi^{2}\right]i
\end{array}
\right\rbrace}{1-\xi-\xi^{2}-2\xi^{3}}.
$$
\end{theorem}
\begin{proof}
Let $g\left(\mathcal{J}g_{n}^{(3)};\xi\right)$ the generating function for the generalized Gaussian third-order Jacobsthal numbers $\mathcal{J}g_{n}^{(3)}$ in the variable $\xi$, in other words
$$
g\left(\mathcal{J}g_{n}^{(3)};\xi\right)=\mathcal{J}g_{0}^{(3)}+\mathcal{J}g_{1}^{(3)}\xi+\mathcal{J}g_{2}^{(3)}\xi^{2}+\cdots .
$$
Thus, we have
\begin{align*}
g\left(\mathcal{J}g_{n}^{(3)};\xi\right)\left[1-\xi-\xi^{2}-2\xi^{3}\right]&=\mathcal{J}g_{n}^{(3)}\\
&\ \ +(\mathcal{J}g_{1}^{(3)}-\mathcal{J}g_{0}^{(3)})\xi\\
&\ \ +(\mathcal{J}g_{2}^{(3)}-\mathcal{J}g_{1}^{(3)}-\mathcal{J}g_{0}^{(3)})\xi^{2}\\
&\ \ +(\mathcal{J}g_{3}^{(3)}-\mathcal{J}g_{2}^{(3)}-\mathcal{J}g_{1}^{(3)}-2\mathcal{J}g_{0}^{(3)})\xi^{3}\\
&\ \ +\ldots \\
&=\mathcal{J}g_{n}^{(3)}\\
&\ \ +(\mathcal{J}g_{1}^{(3)}-\mathcal{J}g_{0}^{(3)})\xi\\
&\ \ +(\mathcal{J}g_{2}^{(3)}-\mathcal{J}g_{1}^{(3)}-\mathcal{J}g_{0}^{(3)})\xi^{2}.
\end{align*}
Using the initial values of sequence $\{\mathcal{J}g_{n}^{(3)}\}_{n\geq0}$, we obtain the requested result.
\end{proof}

The following corollary gives the generating function for the Gaussian third-order Jacobsthal and Gaussian modified third-order Jacobsthal numbers.
\begin{corollary}
For $1-\xi-\xi^{2}-2\xi^{3}\neq 0$, we have
$$
g\left(Jg_{n}^{(3)};\xi\right)=\frac{\xi+\xi^{2}i}{1-\xi-\xi^{2}-2\xi^{3}}.
$$
and
$$
g\left(Kg_{n}^{(3)};\xi\right)=\frac{3-2\xi-\xi^{2}-\frac{1}{2}\left[1-7\xi+3\xi^{2}\right]i}{1-\xi-\xi^{2}-2\xi^{3}}.
$$
\end{corollary}

The next theorem deal with the finite sums of the generalized Gaussian third-order Jacobsthal numbers. Furthermore, the results for Gaussian third-order Jacobsthal and Gaussian modified third-order Jacobsthal numbers are given in the subsequent corollary.

\begin{theorem}
For every non-negative integer $n$, the following sum formula is satisfied
\begin{equation}\label{s1}
\sum_{l=0}^{n}\mathcal{J}g_{l}^{(3)}=\frac{1}{3}\left[\mathcal{J}g_{n+2}^{(3)}+2\mathcal{J}g_{n}^{(3)}+(a-c)+\frac{1}{2}(c-3b-a)i\right].
\end{equation}
\end{theorem}
\begin{proof}
We will use mathematical induction on $n$. First if $n=0$, we have
$$
\mathcal{J}g_{0}^{(3)}=\frac{1}{3}\left[\mathcal{J}g_{2}^{(3)}+2\mathcal{J}g_{0}^{(3)}+(a-c)+\frac{1}{2}(c-3b-a)i\right]
=a+\frac{1}{2}(c-b-a)i, 
$$
and the result is true. Now, suppose the statement is true for $n$, we will prove $n+1$
\begin{align*}
\sum_{l=0}^{n+1}\mathcal{J}g_{l}^{(3)}&=\sum_{l=0}^{n}\mathcal{J}g_{l}^{(3)}+\mathcal{J}g_{n+1}^{(3)}\\
&=\frac{1}{3}\left[\mathcal{J}g_{n+2}^{(3)}+2\mathcal{J}g_{n}^{(3)}+(a-c)+\frac{1}{2}(c-3b-a)i\right]+\mathcal{J}g_{n+1}^{(3)}\\
&=\frac{1}{3}\left[\mathcal{J}g_{n+2}^{(3)}+2\mathcal{J}g_{n}^{(3)}+(a-c)+\frac{1}{2}(c-3b-a)i+3\mathcal{J}g_{n+1}^{(3)}\right]\\
&=\frac{1}{3}\left[\mathcal{J}g_{n+3}^{(3)}+2\mathcal{J}g_{n+1}^{(3)}+(a-c)+\frac{1}{2}(c-3b-a)i\right].
\end{align*}
This proves what was requested.
\end{proof}

\begin{corollary}
For every non-negative integer $n$, we have
$$
\sum_{l=0}^{n}Jg_{l}^{(3)}=\frac{1}{3}\left[Jg_{n+2}^{(3)}+2Jg_{n}^{(3)}-1-2i\right]
$$
and
$$
\sum_{l=0}^{n}Kg_{l}^{(3)}=\frac{1}{3}\left[Kg_{n+2}^{(3)}+2Kg_{n}^{(3)}-\frac{3}{2}i\right].
$$
\end{corollary}

\section{$q$-Generalized Gaussian third-order Jacobsthal numbers}
In this section, we present a new family that generalizes the generalized Gaussian third-order Jacobsthal numbers.  This new sequence will be called $q$-Generalized Gaussian third-order Jacobsthal numbers and we will study here some of its algebraic properties. This idea is taken from El-Mikkawy and Sogabe's work in \cite{El}.

\begin{definition}\label{def3}
Let $n$ be a non-negative integer and $q$ be a positive integer. Then, there exist unique $a$ and $r$ such that $n=aq+r$ and $0\leq r<q$. Using these parameters, the $q$-generalized Gaussian third-order Jacobsthal numbers $\{\mathcal{J}g_{n}^{(3)}(q)\}_{n\geq0}$ are defined as
$$
\mathcal{J}g_{n}^{(3)}(q)=\frac{1}{7^{q}}\left[\lambda\left(1+\frac{i}{2}\right)2^{a}+Z_{a}\right]^{q-r}\cdot\left[\lambda\left(2+i\right)2^{a}+Z_{a+1}\right]^{r},
$$
where $Z_{a}=\left[4a+4b-3c+(2a-5b+2c)i\right]\Omega_{a}+ \left[6a-b-c-(4a+4b-3c)i\right]\Omega_{a+1}$ and  $\lambda=a+b+c$.
\end{definition}

From Theorem \ref{teo1} and Definition \ref{def3}, the relation between $q$-generalized Gaussian third-order Jacobsthal and generalized Gaussian third-order Jacobsthal numbers is given by
\begin{equation}\label{im}
\mathcal{J}g_{aq+r}^{(3)}(q)=\left[\mathcal{J}g_{a}^{(3)}\right]^{q-r}\left[\mathcal{J}g_{a+1}^{(3)}\right]^{r},\ \ 0\leq r<q.
\end{equation}

\begin{remark}
If $q=1$, then $r=0$. From the above equation, it follows that $n=a$ and $\mathcal{J}g_{a}^{(3)}(1)=\mathcal{J}g_{a}^{(3)}$. In particular, using equation (\ref{im}) and $r=0$, we get
$$
\mathcal{J}g_{aq}^{(3)}(q)=\left[\mathcal{J}g_{a}^{(3)}\right]^{q}.
$$
\end{remark}

\begin{example}
We give some particular cases of Definition \ref{def3} for some values of $q$:
\begin{equation}
q=2:\ \left\{ 
\begin{array}{lll}
\mathcal{J}g_{2a}^{(3)}(2)=\left[\mathcal{J}g_{a}^{(3)}\right]^{2}  \\ 
\mathcal{J}g_{2a+1}^{(3)}(2)=\left[\mathcal{J}g_{a}^{(3)}\right]\left[\mathcal{J}g_{a+1}^{(3)}\right]  
\end{array}
\right. ,
\end{equation}
\begin{equation}
q=3:\ \left\{ 
\begin{array}{lll}
\mathcal{J}g_{3a}^{(3)}(3)=\left[\mathcal{J}g_{a}^{(3)}\right]^{3}  \\ 
\mathcal{J}g_{3a+1}^{(3)}(3)=\left[\mathcal{J}g_{a}^{(3)}\right]^{2}\left[\mathcal{J}g_{a+1}^{(3)}\right]  \\
\mathcal{J}g_{3a+2}^{(3)}(3)=\left[\mathcal{J}g_{a}^{(3)}\right]\left[\mathcal{J}g_{a+1}^{(3)}\right]^{2} \end{array}
\right. 
\end{equation}
and
\begin{equation}
q=4:\ \left\{ 
\begin{array}{lll}
\mathcal{J}g_{4a}^{(3)}(4)=\left[\mathcal{J}g_{a}^{(3)}\right]^{4}  \\ 
\mathcal{J}g_{4a+1}^{(3)}(4)=\left[\mathcal{J}g_{a}^{(3)}\right]^{3}\left[\mathcal{J}g_{a+1}^{(3)}\right]  \\
\mathcal{J}g_{4a+2}^{(3)}(4)=\left[\mathcal{J}g_{a}^{(3)}\right]^{2}\left[\mathcal{J}g_{a+1}^{(3)}\right]^{2} \\
\mathcal{J}g_{4a+3}^{(3)}(4)=\left[\mathcal{J}g_{a}^{(3)}\right]\left[\mathcal{J}g_{a+1}^{(3)}\right]^{3} 
\end{array}
\right. 
\end{equation}
\end{example}

The following result establishes a new connection between these sequences.
\begin{theorem}
Let $a\geq 1$ be an integer. For an integer $q\geq 2$, we have
$$
\mathcal{J}g_{aq+1}^{(3)}(q)=2\mathcal{J}g_{aq}^{(3)}(q)-\mathcal{J}g_{a(q-1)}^{(3)}(q-1)\Xi_{a}
$$
and
$$
\mathcal{J}g_{aq+1}^{(3)}(q)=2\mathcal{J}g_{aq}^{(3)}(q)+\mathcal{J}g_{(a+1)q-1}^{(3)}(q) +\mathcal{J}g_{a(q-1)}^{(3)}(q-1)\Xi_{a},
$$
where $\Xi_{a}=\left[2a+b-c+(2b-c)i\right]\Omega_{a}+\left[2a-b-(2a+b-c)i\right]\Omega_{a+1}$.
\end{theorem}
\begin{proof}
From equations (\ref{im}) and (\ref{teo2:3}), we have
\begin{align*}
\mathcal{J}g_{aq+1}^{(3)}(q)&=\left[\mathcal{J}g_{a}^{(3)}\right]^{q-1}\left[\mathcal{J}g_{a+1}^{(3)}\right]\\
&=\left[\mathcal{J}g_{a}^{(3)}\right]^{q-1}\left[2\mathcal{J}g_{a}^{(3)}-\Xi_{a}\right]\\
&=2\left[\mathcal{J}g_{a}^{(3)}\right]^{q-1}\left[\mathcal{J}g_{a}^{(3)}\right]-\left[\mathcal{J}g_{a}^{(3)}\right]^{q-1}\left[\Xi_{a}\right]\\
&=2\mathcal{J}g_{aq}^{(3)}(q)-\mathcal{J}g_{a(q-1)}^{(3)}(q-1)\Xi_{a},
\end{align*}
where $\Xi_{a}=\left[2a+b-c+(2b-c)i\right]\Omega_{a}+\left[2a-b-(2a+b-c)i\right]\Omega_{a+1}$. Similarly, the other result is proved.
\end{proof}

Using Theorem \ref{cas} and $\left[\mathcal{J}g_{n}^{(3)}\right]^{2}=\mathcal{J}g_{2n}^{(3)}(2)$, we can write
\begin{align*}
\mathcal{J}g_{2n}^{(3)}(2)&-\mathcal{J}g_{n-1}^{(3)}\mathcal{J}g_{n+1}^{(3)}\\
&=\frac{1}{49}\left\lbrace
\begin{array}{c}
\lambda_{1}^{2}-\lambda_{1}\lambda_{2}+\lambda_{2}^{2}\\
+\lambda\left(1+\frac{i}{2}\right)2^{n-1}\left[(8\lambda_{1}-3\lambda_{2})\Omega_{n}+(3\lambda_{1}+5\lambda_{2})\Omega_{n+1}\right]
\end{array}
\right\rbrace .
\end{align*}

\begin{theorem}
For a fixed non-negative integer $a$ and $q$, the following result follows
$$
\sum_{r=0}^{q-1}\binom{q-1}{r}\mathcal{J}g_{aq+r}^{(3)}(q)=\mathcal{J}g_{a}^{(3)}\left[3\mathcal{J}g_{a}^{(3)}-\Xi_{a}\right]^{q-1},
$$
where $\Xi_{a}=\left[2a+b-c+(2b-c)i\right]\Omega_{a}+\left[2a-b-(2a+b-c)i\right]\Omega_{a+1}$.
\end{theorem}
\begin{proof}
Using equations (\ref{im}) and (\ref{teo2:3}), we get
\begin{align*}
\sum_{r=0}^{q-1}\binom{q-1}{r}\mathcal{J}g_{aq+r}^{(3)}(q)&=\sum_{r=0}^{q-1}\binom{q-1}{r}\left[\mathcal{J}g_{a}^{(3)}\right]^{q-r}\left[\mathcal{J}g_{a+1}^{(3)}\right]^{r}\\
&=\mathcal{J}g_{a}^{(3)}\sum_{r=0}^{q-1}\binom{q-1}{r}\left[\mathcal{J}g_{a}^{(3)}\right]^{q-1-r}\left[\mathcal{J}g_{a+1}^{(3)}\right]^{r}\\
&=\mathcal{J}g_{a}^{(3)}\left[\mathcal{J}g_{a}^{(3)}+\mathcal{J}g_{a+1}^{(3)}\right]^{q-1}\\
&=\mathcal{J}g_{a}^{(3)}\left[3\mathcal{J}g_{a}^{(3)}-\Xi_{a}\right]^{q-1},
\end{align*}
where $\Xi_{a}=\left[2a+b-c+(2b-c)i\right]\Omega_{a}+\left[2a-b-(2a+b-c)i\right]\Omega_{a+1}$. The result is proven.
\end{proof}

\begin{theorem}
For a fixed non-negative integer $a$ and $q$, the following result follows
$$
\sum_{r=0}^{q-1}\mathcal{J}g_{aq+r}^{(3)}(q)=\mathcal{J}g_{a}^{(3)}\frac{\left[\mathcal{J}g_{a+1}^{(3)}\right]^{q}-\left[\mathcal{J}g_{a}^{(3)}\right]^{q}}{\mathcal{J}g_{a}^{(3)}-\Xi_{n}},
$$
where $\Xi_{a}=\left[2a+b-c+(2b-c)i\right]\Omega_{a}+\left[2a-b-(2a+b-c)i\right]\Omega_{a+1}$.
\end{theorem}
\begin{proof}
Using equations (\ref{im}) and (\ref{teo2:3}), we get
\begin{align*}
\sum_{r=0}^{q-1}\mathcal{J}g_{aq+r}^{(3)}(q)&=\sum_{r=0}^{q-1}\left[\mathcal{J}g_{a}^{(3)}\right]^{q-r}\left[\mathcal{J}g_{a+1}^{(3)}\right]^{r}\\
&=\left[\mathcal{J}g_{a}^{(3)}\right]^{q}\sum_{r=0}^{q-1}\left[\frac{\mathcal{J}g_{a+1}^{(3)}}{\mathcal{J}g_{a}^{(3)}}\right]^{r}\\
&=\left[\mathcal{J}g_{a}^{(3)}\right]^{q}\frac{\left[\frac{\mathcal{J}g_{a+1}^{(3)}}{\mathcal{J}g_{a}^{(3)}}\right]^{q}-1}{\left[\frac{\mathcal{J}g_{a+1}^{(3)}}{\mathcal{J}g_{a}^{(3)}}\right]-1}.
\end{align*}
After an algebraic calculation, we have
\begin{align*}
\left[\mathcal{J}g_{a}^{(3)}\right]^{q}\frac{\left[\frac{\mathcal{J}g_{a+1}^{(3)}}{\mathcal{J}g_{a}^{(3)}}\right]^{q}-1}{\left[\frac{\mathcal{J}g_{a+1}^{(3)}}{\mathcal{J}g_{a}^{(3)}}\right]-1}&=\mathcal{J}g_{a}^{(3)}\frac{\left[\mathcal{J}g_{a+1}^{(3)}\right]^{q}-\left[\mathcal{J}g_{a}^{(3)}\right]^{q}}{\mathcal{J}g_{a+1}^{(3)}-\mathcal{J}g_{a}^{(3)}}\\
&=\mathcal{J}g_{a}^{(3)}\frac{\left[\mathcal{J}g_{a+1}^{(3)}\right]^{q}-\left[\mathcal{J}g_{a}^{(3)}\right]^{q}}{\mathcal{J}g_{a}^{(3)}-\Xi_{n}}.
\end{align*}
The result is proven.
\end{proof}

\section{Conclusions}
In this study, we introduced the generalized Gaussian third-order Jacobsthal numbers and studied their algebraic properties and several of its identities. Further, we investigated a new special family of $q$-generalized Gaussian third-order Jacobsthal numbers in closed form and shown some relations with usual Gaussian third-order Jacobsthal numbers with arbitrary initial values. Suggestions for future research:
\begin{enumerate}
\item Application to other number sequences: Future studies could explore the application of generalized Gaussian numbers to other special sequences, such as tribonacci or Lichtenberg, to investigate whether similar properties and patterns emerge.
\item Generalization to higher dimensions: Extending the generalized Gaussian numbers framework to dual numbers or hyper-dual numbers could provide deeper insights and broader applications in fields like derivative calculations and multi-body kinematics.
\item Connections with coding theory and cryptography: The generalized Gaussian matrices could be studied in the context of coding theory, particularly for analyzing new codes and ciphers.
\end{enumerate}
By building on the foundations laid in this work, further research can uncover additional properties and applications of generalized Gaussian numbers and matrices, cementing their role in both theoretical studies and applied mathematics.

\medskip


\begin{thebibliography}{99}

\bibitem{As} 
M. A\c{s}ci, E. G\"urel, 
\emph{Gaussian Jacobsthal and Gaussian Jacobsthal Lucas numbers},
Ars Comb.,
\textbf{111} (2013), 53--63.

\bibitem{Be} 
G. Berzsenyi, 
\emph{Gaussian Fibonacci numbers},
Fibonacci Quart.,
\textbf{15}(2) (1977), 233--236.

\bibitem{El} 
M. El-Mikkawy, T. Sogabe,  
\emph{A new family of $k$-Fibonacci numbers},
Applied Mathematics
and Computation,
\textbf{215} (2010), 4456--4461.

\bibitem{Ha} 
S. Halici, S. Oz, 
\emph{On Gaussian Pell polynomials and their some properties},
Palest. J. Math.,
\textbf{7}(1) (2018), 251--256.

\bibitem{Jo} 
J.H. Jordan,  
\emph{Gaussian Fibonacci and Lucas numbers},
Fibonacci Quart.,
\textbf{3}(4) (1965), 315--318.

\bibitem{Ku} 
M. Kumari, K. Prasad, J. Tanti,
\emph{Generalization of Gaussian Mersenne numbers and their new families},
Bull. Transilv. Univ. Bra\c{s}ov, Ser. III, Math. Comput. Sci.,
\textbf{5}(1) (2025), 147--160.

\bibitem{Mo1} 
G. Morales, 
\emph{On Gauss Third-order Jacobsthal numbers and their applications},
Analele Stiint ale Universitatii Alexandru Ioan Cuza din Iasi.,
\textbf{67}(2) (2021), 231--241.

\bibitem{Mo2} 
G. Morales,
\emph{Gaussian third-order Jacobsthal and Gaussian third-order Jacobsthal-Lucas polynomials and their properties},
Asian-European Journal of Mathematics,
\textbf{14}(5) (2021), 1--13.

\bibitem{Mo3} 
G. Morales,
\emph{On third-order Jacobsthal polynomials and their properties},
Asian-European Journal of Mathematics,
\textbf{22}(1) (2021), 123--132.

\bibitem{Mo4} 
G. Morales,
\emph{On bicomplex third-order Jacobsthal numbers},
Complex Variables and Elliptic Equations,
\textbf{68}(1) (2023), 44--56.

\bibitem{Mo5} 
G. Morales,
\emph{A note on modified third-order Jacobsthal quaternions and their properties},
Acta et Commentationes Universitatis Tartuensis de Mathematica,
\textbf{28}(2) (2024), 187--196.

\bibitem{Mo6} 
G. Morales,
\emph{Binomial transforms of the third-order Jacobsthal and modified third-order Jacobsthal polynomials},
Universal Journal of Mathematics and Applications,
\textbf{7}(3) (2024), 144--151.


\bibitem{Pe} 
S. Pethe, A.F. Horadam,  
\emph{Generalised Gaussian Fibonacci numbers},
Bull. Aust. Math. Soc.,
\textbf{33}(1) (1986), 37--48.

\bibitem{Ta} 
D. Ta\c{s}ci,
\emph{Gaussian Padovan and Gaussian Pell-Padovan sequences},
Commun. Fac. Sci. Univ. Ank., S\'er. A1, Math. Stat.,
\textbf{67}(2) (2018), 82--88.

\end{thebibliography}
\end{document}